\documentclass[12pt,reqno]{amsart}
\usepackage{amssymb,bm}
\usepackage{graphicx}
\usepackage{amsmath}
\usepackage{a4wide}
\usepackage{version}
\usepackage{amsfonts}
\usepackage{enumerate}
\usepackage{comment}
\usepackage{epsfig}
\usepackage{pdfpages}
\usepackage{amsthm}
\usepackage{longtable}
\usepackage{import}
\usepackage{enumitem}  
\usepackage{url}  
\usepackage{mathrsfs}  

\title{The Cassels heights of cyclotomic integers}
\author{James McKee}
\address{Department of Mathematics, Royal Holloway, University of London, Egham Hill,  Egham, Surrey, TW20 0EX, U.K.}
\email{james.mckee@rhul.ac.uk}

\author{Byeong-Kweon Oh}
\address{Department of Mathematical Sciences and
Research Institute of Mathematics, Seoul National University Seoul 08826, Republic of Korea} 
\email{bkoh@snu.ac.kr}

\author{Chris Smyth}
\address{School of Mathematics\\
University of Edinburgh\\
Edinburgh EH9 3FD\\
Scotland, U.K.}
\email{c.smyth@ed.ac.uk}

\begin{document}

\maketitle

\bibliographystyle{plain}
\newcommand{\fl}[1]{\left\lfloor #1\right\rfloor}
\newcommand{\qr}[2]{\left(\frac{#1}{#2}\right)}
\newcommand{\ep}[2]{\equiv{#1}\pmod{#2}}
\def\T{\textup{\textsf{T}}}

\newtheorem{theorem}{Theorem}
\newtheorem{acknowledgement}[theorem]{Acknowledgement}
\newtheorem{cor}[theorem]{Corollary}
\newtheorem{lemma}[theorem]{Lemma}
\newtheorem{prop}[theorem]{Proposition}
\newtheorem{algorithm}[theorem]{Algorithm}
\newtheorem{case}[theorem]{Case}
\newtheorem{claim}[theorem]{Claim}
\newtheorem{conclusion}[theorem]{Conclusion}
\newtheorem{condition}[theorem]{Condition}
\newtheorem{criterion}[theorem]{Criterion}
\newtheorem{conj}[theorem]{Conjecture}
\newtheorem{notation}[theorem]{Notation}
\newtheorem{solution}[theorem]{Solution}
\newtheorem{summary}[theorem]{Summary}

\theoremstyle{definition}
\newtheorem{definition}[theorem]{Definition}
\newtheorem*{rmk}{Remark}

\newtheorem{ex}[theorem]{Exercise}
\newtheorem{problem}[theorem]{Problem}
\newtheorem{oproblem}[theorem]{Open Problem}
\newtheorem{eex}[theorem]{Easy exercise}
\theoremstyle{remark}

\newtheoremstyle{dotless}{}{}{\itshape}{}{\bfseries}{}{ }{}
\theoremstyle{dotless}
\newtheorem*{Rprob}{Robinson's Problem}
\newtheorem*{Rconj}{Robinson's Conjecture}

\makeatletter
\@namedef{subjclassname@2010}{
  \textup{2010} Mathematics Subject Classification}
\makeatother

\def\H{\mathcal H}
\def\la{\lambda}
\def\l{\ell}
\def\Z{\mathbb Z}
\def\N{\mathbb N}
\def\R{\mathbb R}
\def\bee{\begin{enumerate}}\def\ene{\end{enumerate}}
\def\bei{\begin{itemize}}\def\eni{\end{itemize}}
\def\beq{\begin{equation}}\def\enq{\end{equation}}
\def\bepr{\begin{proof}}\def\enpr{\end{proof}}
\def\bet{\begin{theorem}}\def\ent{\end{theorem}}
\def\bel{\begin{lemma}}\def\enl{\end{lemma}}

\def\i{\item}
\def\bysame{---}

\def\C{\mathbb C}\def\D{\mathbb D}\def\F{\mathbb F}
\def\G{\mathbb G}
\def\Gmn{{\mathbb G}_m^n}
\def\N{\mathbb N}
\def\R{\mathbb R}\def\RR{\mathcal R}
\def\Q{\mathbb Q}\def\T{\mathcal T}\def\Tr{\textup{\textsf{T}}}
\def\Z{\mathbb Z}
\def\L{\mathcal L}
\def\al{\alpha}
\def\be{\beta}
\def\ga{\gamma}
\def\u{\mathbf u}\def\uu{\underline{u}}\def\ua{\underline{a}}
\def\eps{\varepsilon}
\def\de{\delta}
\def\ro{\rho}
\def\phi{\varphi}
\def\om{\omega}
\def\si{\sigma}
\def\th{\theta}
\def\bpsi{\boldsymbol\psi}
\def\cc {,\dots,}
\def\bi{\mathbf i}
\def\bom{\boldsymbol \omega}
\def\bla{\boldsymbol \lambda}
\def\bt{\mathbf t}
\def\br{\mathbf r}
\def\bx{\mathbf x}
\def\bo{\mathbf 0}
\def\ba{\mathbf a}
\def\D{\mathcal D}
\def\r{\mathbf r}\def\y{\mathbf y}\def\z{\mathbf z}\def\0{\mathbf 0}\def\s{\mathbf s}
\def\bc{\mathbf c}
\def\x{\mathbf x}\def\j{\mathbf j}\def\w{\mathbf w}
\def\Re{\operatorname{Re}}\def\Im{\operatorname{Im}}
\def\GL{\operatorname{GL}}
\def\Gal{\operatorname{Gal}}
\def\tr{\overline{\operatorname{tr}}}
\def\trace{\operatorname{trace}}
\def\sgn{\operatorname{sgn}}
\def\ind{\operatorname{ind}}
\def\lcm{\operatorname{lcm}}
\def\res{\operatorname{res}}
\def\exp{\operatorname{exp}}
\def\diag{\operatorname{diag}}
\def\deg{\operatorname{deg}}
\def\max{\operatorname{max}}
\def\dim{\operatorname{dim}}
\def\BM{\operatorname{BM}}
\def\Norm{\operatorname{Norm}}

\def\l{\ell}
\def\M{\mathscr M}
\def\CC{\mathscr C}
\def\NC{\mathscr N}
\def\MC{\mathscr M}
\def\t{\theta}
\def\threesidedbox#1{\setbox0=\hbox{$#1$}\dimen0=\wd0 \advance\dimen0 by3pt\rlap{\hbox{\vrule height8pt width.4pt depth2pt \kern-.4pt\vrule height8.4pt width\dimen0 depth-8pt\kern-.4pt \vrule height8pt width.4pt depth2pt}}\relax \hbox to\dimen0{\hss$#1$\hss}}
\def\oldho#1{\threesidedbox#1}

\makeatletter

\newcommand*\strictceil[1]{%
    \mathord{\mathpalette\@strictceil{#1}}%
}
\newcommand*\@strictceil[2]{%
    \vbox{\m@th
        \dimen@ \fontdimen 8
            \ifx\scriptscriptstyle #1%
                \scriptscriptfont
            \else\ifx\scriptstyle #1%
                \scriptfont
            \else
                \textfont
            \fi\fi \thr@@
        \kern \dimen@ 
        \hbox{%
            \vrule \@width\dimen@
            \vbox{%
                \kern -\dimen@
                \hbox{$#1\overline{%
                    \kern \thr@@\dimen@
                    \begingroup
                        #2
                    \endgroup
                    \kern \thr@@\dimen@
                }$}%
            }%
            \vrule \@width\dimen@
        }%
    }%
}

\makeatother
\def\ho#1{\strictceil#1}

\begin{abstract} We study the set $\CC$ of mean square values of the moduli of the conjugates of cyclotomic integers $\be$. For its $k$th derived set $\CC^{(k)}$, we show that  $\CC^{(k)}=(k+1)\CC\,\, (k\ge 0)$, so that also  $\CC^{(k)}+\CC^{(\l)}=\CC^{(k+\l+1)}\,\,(k,\l\ge 0)$. We also calculate the order type of $\CC$, and show that it is the same as that of the set of PV numbers.

Furthermore, we describe precisely the restricted set $\CC_p$ where the $\be$ are confined to the ring $\Z[\om_p]$, where $p$ is an odd prime and $\om_p$ is a primitive $p$th root of unity.
In order to do this, we prove that both of the quadratic polynomials $a^2+ab+b^2+c^2+a+b+c$
and $a^2+b^2+c^2+ab+bc+ca+a+b+c$ are universal.
\end{abstract}
\section{Introduction}\label{S-defs}

A {\it cyclotomic integer} is an algebraic integer $\be$ that can be written as a sum of roots of unity. 
Any such $\beta$ lies in $\mathbb{Z}[\om_n]$ for some $n$, where $\om_n$ is a primitive $n$th root of unity, and it is well known that $\Z[\om_n]$ is the ring of integers of the field $\Q(\om_n)$. 
If $\be_1=\be,\be_2\cc \be_n$ are the Galois conjugates of $\be$ (or indeed a list that includes each Galois conjugate the same number of times), we define, following Cassels \cite{Cassels1969}, $\M(\be)$ by
\[
\M(\be)=\frac1{n}\sum_{j=1}^n|\be_j|^2.
\]
Let us call this value the {\it Cassels height of $\be$}. Because, as first noted by Robinson \cite{Robinson1965}, the $|\be_j|^2$ are the 
conjugates of $|\be|^2$ (something that is not true for algebraic integers generally), $\M(\be)$ is rational. From the AM-GM inequality it follows immediately that $\M(\be)\ge 1$ for $\be\ne 0$.
Two nonzero cyclotomic integers are said to be {\it equivalent} if dividing the first by some conjugate of the second gives a root of unity.
Equivalent cyclotomic integers have the same Cassels height.

The aim of this paper is to study the set
\[
\CC=\{ \M(\be)\mid \be \text{ a nonzero cyclotomic integer}\}.
\]
This set has an interesting structure.
 In 2009 Stan and Zaharescu \cite[Theorem 4]{StanZaharescu2009}  proved the following results concerning $\CC$:
\bei
\item[(i)]{\bf Closure.} The set $\CC$ is a closed subset of $\Q$.  (See also \cite[Theorem 9.1.1]{CalegariMorrisonSnyder2011}).
\item[(ii)] {\bf Additivity.} The set $\CC$ is closed under addition. (This also follows from (i) and  Proposition \ref{P-2lim} below.)
\item[(iii)]  For every rational number $r\in[0,1)$ there is an integer $n_0$ such that $r+n\in\CC$ for all $n\ge n_0$.
\eni
They applied their results to deducing facts about character values of finite groups.
We extend (i) and (ii) to obtain the following results, connecting the $k$th derived set $\CC^{(k)}$ of $\CC$ (the derived set of $\CC^{(k-1)}$, with $\CC^{(0)}=\CC)$ and the Minkowski sumset 
\beq\label{E-sumset}
k\CC=\{c_1+c_2+\cdots+c_k\mid c_1,c_2\cc c_k\in\CC\}.
\enq

\begin{theorem} \label{T-1} For $k\ge 1$ the $k$th derived set $\CC^{(k)}$ of $\CC$ is equal to the sumset $(k+1)\CC$.
Furthermore every element of $\CC^{(k)}$ is a limit from both sides of elements of $\CC^{(k-1)}$.
\end{theorem}
The following is an immediate consequence.
\begin{cor}\label{C-1} The smallest element of $\CC^{(k)}\,\,(k\ge 0)$ is $k+1$. Furthermore,
a stronger version of additivity holds, namely that  $\CC^{(k)}+\CC^{(\l)}=\CC^{(k+\l+1)}\,\,(k,\l\ge 0)$.
\end{cor}

Sets having similar topological (though not algebraic) structure as $\CC$ have been found before. Salem \cite{Salem1944} proved that the set $S$ of all Pisot-Vijayaraghavan (PV) numbers is closed in $\R$.
The sets $S^{(k)}$ are known to be nonempty, with the smallest element being at least $\sqrt{k}$ -- see \cite{Boyd1979}. Also, Boyd and Mauldin in 1996  \cite{BoydMauldin1996} proved that for $k\ge 1$ every member of $S^{(k)}$ is a limit from both sides of elements of $S^{(k-1)}$. This enabled them to specify the order type of $S$. With this in mind, and recalling that Axel Thue \cite{Thue1912} was the discoverer of the PV numbers, we define a {\it Thue set} $T$ to be a subset of the positive real line with the following properties:
\bei\item[(i)] The set $T$ is a closed subset of $\R_+$;
\item[(ii)] For $k\ge 1$  the $k$th derived set $T^{(k)}$ is nonempty, and every element of it is a limit from both sides of elements of $T^{(k-1)}$;
\item[(iii)]  $t_k:=\min\{t\mid t\in T^{(k)}\} \to\infty$ as $k\to\infty$.
\eni
So $S$ is a Thue set.
\begin{cor}\label{C-2} The set $\CC$ is a Thue set.
\end{cor}
It is immediately clear that all derived sets of a Thue set are again Thue sets.
Thus all the derived sets $\CC^{(k)}$ for $k\ge 1$ are also Thue sets.


It may be that the set of all Mahler measures of polynomials in any number of variables and having integer coefficients also forms a Thue set. Boyd \cite{Boyd1981b} conjectured that this set is closed. There is some further evidence for the set being a Thue set  in \cite{Smyth2018}.

Our second main result concerns the set of those $\M(\be)$ where, for a given odd prime $p$,  $\be$ is a sum of $2p$th roots of unity. We denote this set by $\CC_p$, so that 
\[
\CC_p=\{\M(\be)\mid \be\in\Z[\om_p]\},
\]
where $\om_p$ is a primitive $p$th root of unity.

\begin{theorem} \label{T-2}  For all primes $p\ge 5$ the set $\CC_p$ is given by
\beq\label{E-poss-form}
\CC_p=\left\{\frac1{p'}\left(\tfrac12 s(p-s)+r p\right)\mid s=0,1\cc p' \text{ and } r\ge 0\right\}.
\enq
Here $p':=(p-1)/2$.
\end{theorem}
It is easy to check that the elements specified by \eqref{E-poss-form} are all distinct.

For $p=3$ the set $\CC_3$ a proper subset of the set given by the RHS of \eqref{E-poss-form}.  
Indeed $\CC_3$ is easily seen to be the set of integers of the form
$(a+b\om_3)(a+b\om_3^2)=a^2-ab+b^2$, namely all integers $N$ with prime factorisation of the form $N=\prod_q q^{e_q}$, where $e_q$ is even for all primes $q\equiv 2\pmod{3}$. However for $p=3$ the set on the RHS of \eqref{E-poss-form} consists of all integers $N\not\equiv 2\pmod{3}$. So for instance $6,10,15$ and $18$ belong to this set, but do not belong to $\CC_3$.

For the proof in the case $p=5$ we need to prove the universality of two ternary quadratic polynomials.

\begin{theorem}\label{T-quads} Both of the quadratic polynomials
\beq\label{E-qp1}
a^2+ab+b^2+c^2+a+b+c
\enq
and
\beq\label{E-qp2}
a^2+b^2+c^2+ab+bc+ca+a+b+c
\enq
represent all positive integers for integer values of their variables (i.e., they are {\bf universal}).

\end{theorem}

Of course it would be interesting to study $\CC_n:=\{\M(\be)\mid \be\in\Z[\om_n]\}$ for $n$ composite, too.

\subsection{Background} The study of cyclotomic integers began in earnest with a paper of Raphael Robinson in 1965 \cite{Robinson1965}. In it he stated two problems and proposed five conjectures about them. Schinzel \cite{Schinzel1966} solved his second problem and proved his third conjecture. In 1968 Jones \cite{Jones1968} proved Robinson's fifth conjecture. Cassels solved Robinson's Conjecture 2 in \cite{Cassels1969}, with the help of his $\M$ function. Loxton \cite{Loxton1974} solved Robinson's first problem, and also improved on Schinzel's solution of the second problem. In 2013 F. Robinson and M. Wurtz \cite{RobinsonWurtz2013} proved Robinson's fourth conjecture. (They also said that the first conjecture had been proved, although this does not seem to be the case.)

Cassels \cite{Cassels1969} also showed that the only $\M(\be)<2$ were for $\be$ that can be written as a sum of at most two roots of unity; this implies that $2$ is the smallest limit point of $\CC$.

In 2011 Calegari, Morrison and Snyder \cite{CalegariMorrisonSnyder2011} studied cyclotomic integers $\be$ with a view to applications to fusion categories and subfactors. As part of this study (their Theorem 9.0.1) they found all $\be$ with $\M(\be)<9/4$.

\section{Proof of Theorem \ref{T-1}}

For the proof, we need a qualitative version of a very precise theorem of Loxton.
\begin{theorem}[{{\cite[eqn. (6.1)]{Loxton1972}}}]\label{T-Loxg} There is a strictly increasing (concave) function $g$ such that for every cyclotomic integer $\be$ we have $\M(\be)\ge g(\NC(\be))$.
\end{theorem}
Here $\NC(\be)$ is the smallest number of roots of unity whose sum is $\be$.
Thus if $\M(\be)\le B$ then $\NC(\be)\le B'$ for some $B'$.

 For any algebraic integer $\al$ we denote the `mean trace' $(\text{trace }(\al))/[\Q(\al):\Q]$ of $\al$ by $\tr(\al)$. This is the mean of the conjugates of $\al$. So $\M(\be)=\tr(|\be|^2)$.
 We need the following basic property of the mean trace.
 \begin{lemma} \label{L-tradd} For any algebraic numbers $\al,\ga$ we have
 \beq\label{E-tradd}
 \tr(\al+\ga)=\tr(\al)+\tr(\ga).
\enq
\end{lemma}
\begin{proof}
Let $F$ be the normal closure of $\Q(\al,\ga)$.
Then
\[
\tr(\al) = \frac1{[F:\Q]}\sum_{\sigma\in\Gal(F/\Q)}\sigma(\al),
\]
from which, using the corresponding formula for $\be$ and for $\al+\be$, the result follows.
\end{proof}

\begin{lemma}\label{L-nonzerotrace} For $\be$  a nonzero cyclotomic integer, with say $\be\in\Z[\om_n]$,  there is some power $\om_n^i$ of $\om_n$ such that $\om_n^i\be$ has nonzero trace.
  \end{lemma}
  \begin{proof} We can write $\be=\sum_{k=0}^{d-1} a_k\om_n^k$, where the $a_k$ are integers, and $d=\phi(n)$. Then the trace of $\be$ (the sum of its conjugates) is 
  \[
  \frac{d'}{d}\sum_{\substack{ j=1 \\ \gcd(j,n)=1}}^{d-1}\sum_{k=0}^{d-1} a_k\om_n^{jk},
  \]
  where $d'=[\Q(\be):\Q]$. Suppose that the trace  of $\om_n^{i}\be$ is $0$ for all $i=0\cc n-1$. Then the traces of all $a_k\om_n^{-k}\be$ would be $0$, and so the trace of $\sum_{k=0}^{d-1} a_k\om_n^{-k}\be=|\be|^2$ would be $0$. But we know that
   the conjugates, $|\be_j|^2$ say, of $|\be|^2$ are all positive, so its trace is positive.
  \end{proof}

 We use $\mu_\phi(n)$ to denote $\mu(n)/\phi(n)$, where $\mu$ is the M\"obius $\mu$-function, and $\phi$ is the Euler $\phi$-function.  
 Thus the mean trace of $\om_n$ is $\mu_\phi(n)$. 
Lemma \ref{L-tradd} states that the mean trace is additive.
Of course it is not generally multiplicative, but there is a special case where this property too holds.
\begin{lemma}\label{L:mult}
Let $m$, $n$ be coprime integers and let $\alpha\in\mathbb{Q}(\om_n)$.
Then
\begin{itemize}
    \item[(i)] $\tr(\om_m\al) = \tr(\om_m)\tr(\al) = \mu_\varphi(m)\tr(\al)$;
    \item[(ii)] if also $m$ is odd, one still has $\tr(\om_{2m}\al) = \tr(\om_{2m})\tr(\al) = -\mu_\varphi(m)\tr(\al)$, regardless of the parity of $n$.
\end{itemize}
\end{lemma}
\begin{proof}
Since $m$ and $n$ are coprime, $\om_m\om_n$ is a primitive $mn$-th root of unity, and the $\varphi(mn)$ automorphisms of $\Q(\om_m\om_n)$ are defined by $\omega_m\omega_n\mapsto\omega_m^a\omega_n^b$ where $a$ is prime to $m$ and $b$ is prime to $n$.
From this the formula in (i) is immediate.
For (ii), given $m$ is odd one has that $-\omega_{2m}$ is a primitive $m$th root of unity and since $\tr(-\be) = -\tr(\be)$ one deduces (ii) from (i).
\end{proof}

\begin{prop} \label{P-2lim} Let $\mathcal L$ be an infinite  increasing sequence of positive integers, and $\ga_1$ and $\ga_2$ be nonzero cyclotomic integers. Then
\[
\lim_{\substack{\l\to\infty \\ \l\in \mathcal L}} \M(\ga_1+\om_{\l}\ga_2)= \M(\ga_1)+\M(\ga_2).
\]
Also, $\mathcal L$ can be chosen so that infinitely many of the values $\M(\ga_1+\om_{\l}\ga_2)$ are distinct, so that $\M(\ga_1)+\M(\ga_2)$   is a genuine limit point of
the sequence $\{\M(\ga_1+\om_{\l}\ga_2)\}_{\l\in\mathcal L}$. Furthermore, $\mathcal L$ can be chosen so that the limit is approached either from above or from below.
\end{prop}

\begin{proof} Now from Lemma \ref{L-tradd}
\begin{align}
\M(\ga_1+\om_{\l}\ga_2)&=\tr(|\ga_1+\om_{\l}\ga_2|^2) \notag\\
&= \tr(|\ga_1|^2)+  \tr(|\ga_2|^2)+ \tr(\om_{-\l}\ga_1\overline{\ga_2}) + \tr(\om_{\l}\overline{\ga_1}\ga_2) \notag\\ 
&= \M(\ga_1)+\M(\ga_2)+ \tr(\om_{\l}^{-1}\ga_1\overline{\ga_2}) + \tr(\om_{\l}\overline{\ga_1}\ga_2).  \label{E-ga12}
\end{align}
Choosing $n$ so that $\ga_1,\ga_2\in\Q(\om_n)$, with say 
\[
\ga_1\overline{\ga_2}=\sum_k a_k\om_n^k,
\]
we see  that 
\[
\tr(\om_{\l}^{-1}\ga_1\overline{\ga_2}) = \tr(\om_{\l}\overline{\ga_1}\ga_2) =\sum_ka_k\tr(\om_\l\om_n^{-k})=\sum_ka_k\tr(\om_{\l'})=\sum_k a_k\mu_\phi(\l'),
\]
where $\om_\l\om_n^{-k}=\om_{\l'}$, say, where $\l'$ depends on $k$. Since $\l'\to\infty$ as $\l\to\infty$,
and $\mu_\phi(\l')\to 0$ as $\l'\to\infty$, we see that as $\l\to\infty$
\[
\M(\ga_1+\om_{\l}\ga_2)\to \M(\ga_1)+\M(\ga_2),
\]
as claimed.

To ensure that this is a genuine limiting process, we need to have $\M(\ga_1+\om_{\l}\ga_2)\ne \M(\ga_1)+\M(\ga_2)$ for infinitely many values of $\l$.
 We now show that $\L$ can be chosen so that this is true.

From Lemma \ref{L-nonzerotrace}, we can choose an integer $i$ such that $\tr(\om_n^i \overline{\ga_1}\ga_2)\ne 0$. Then also $\tr(\om_n^{-i}\ga_1\overline{\ga_2})\ne 0$. Next, define the $\l$'s by  $\om_\l=\om_{\l^*}\,\om_n^{i}$, where the $\l^*$'s are odd primes not dividing $n$.
Then, using Lemma \ref{L:mult}(i),
\[
\tr(\om_{\l}\overline{\ga_1}\ga_2)=\tr(\om_{\l^*}\,\om_n^i\overline{\ga_1}\ga_2)=\tr(\om_{\l^*})\tr(\om_n^i\overline{\ga_1}\ga_2)=-\frac1{\l^*-1}\tr(\om_n^i\overline{\ga_1}\ga_2),
\]
which is nonzero for all $\l$. Hence, from \eqref{E-ga12}, $\M(\ga_1+\om_{\l}\ga_2)$ tends to  $\M(\ga_1)+\M(\ga_2)$ from either above or below (say, above), depending on the sign of $\tr(\om_n^i\overline{\ga_1}\ga_2)$; it never equals $\M(\ga_1)+\M(\ga_2)$.

Finally, if we replace $\l^*$ by $2\l^*$ in the argument  (and see Lemma \ref{L:mult}(ii)), then $-\tfrac1{\l^*-1}$ is replaced by $\tfrac1{\l^*-1}$, so that  $\M(\ga_1+\om_{\l}\ga_2)$ tends to  $\M(\ga_1)+\M(\ga_2)$ from below.
\end{proof}

Note that Proposition \ref{P-2lim} tells us that $2\CC\subseteq\CC^{(1)}$.

\begin{prop} \label{P-bygum} Let $\ga_0,\ga_1\cc \ga_r$ be fixed cyclotomic integers, and for all $n\ge 1$ define
\[
\be_n:=\ga_0+\ga_1\om_{n_1}+ \ga_2\om_{n_2}+\cdots+\ga_r\om_{n_r},
\]
where $n_1\cc n_r$ are integers each tending to infinity as $n\to\infty$, and such that for all $k,\l$ with $1\le k<\l\le r$ the order of $\om_{n_{\l}}/\om_{n_{k}}$
also tends to infinity as $n\to\infty$. Then the sequence $\{\M(\be_n)\}$ converges, say to $\M(\be)$, with 
\[
\M(\be)=\M(\ga_0)+\M(\ga_1)+\cdots+\M(\ga_r).
\]
\end{prop}
\begin{proof} Now putting $n_0=1$ we have
\[
|\be_n|^2=\sum_{k=0}^r |\ga_k|^2+\sum_{\substack{k,\l=0 \\ k\ne\l}}^r \ga_k\overline{\ga_\l}\frac{\om_{n_k}}{\om_{n_{\l}}}.
\]
Choose an integer $t$ so that all the $\ga_k$ belong to $\Q(\om_t)$.
Then taking the mean trace of this expression we obtain
$\M(\be_n)=\sum_{k=0}^r \M(\ga_k)$ plus a sum of terms of the form $\tr(a\om_t^h \om_{n_k}/\om_{n_{\l}})$, where $a$ and $h$ are integers.  Putting $\om_t^h \om_{n_k}/\om_{n_{\l}}=\om_N$ say, we have 
\[
\tr(a\om_t^h \om_{n_k}/\om_{n_{\l}})=a\mu_\phi(N).
\]
Since $N\to\infty$ as $n\to\infty$ we see that as $n\to\infty$ these terms all tend to $0$, so that $\M(\be_n)\to\sum_{k=0}^r \M(\ga_k)$.
\end{proof}

\begin{prop} \label{P-klim} Let $k\ge 1$. Every element of $(k+1)\CC$ belongs to $\CC^{(k)}$ and is a limit from both sides of elements of $k\CC$.
\end{prop}
\begin{proof}
The case $k=1$ has been done in Proposition \ref{P-2lim}. So take $k\ge 2$ and assume the result is true for $k-1$. For cyclotomic integers $\ga_1\cc\ga_k,\ga_{k+1}$, consider
\[
m_{k+1}:=\M(\ga_1)+\cdots+\M(\ga_k)+\M(\ga_{k+1})\in(k+1)\CC.
\]
By the induction hypothesis, for fixed $\l$ the value 
\[
m_{k,\l}:=\M(\ga_1)+\cdots+\M(\ga_{k-1})+\M(\ga_{k}+\om_\l\ga_{k+1})
\]
belongs to $\CC^{(k-1)}$, and is a limit from above of elements of $\CC_{(k-1)}$. Using  Proposition \ref{P-2lim} again, we see that
$m_{k+1}$ is a limit from above of elements of $k\CC\subseteq \CC^{(k-1
)} $, namely the $m_{k,\l}$, as $\l\to\infty$, for $\l$ in some sequence $\mathcal L$. Hence $m_{k+1}\in\CC^{(k)}$. Since we can replace `above' by `below' in the two previous sentences, this proves the result for $k$.

\end{proof}
So certainly the $k$th derived set  $\CC^{(k)}$ of $\CC$ contains $(k+1)\CC$. We need to show that in fact equality holds.

\begin{proof}[Proof of Theorem \ref{T-1}] The theorem holds trivially for $k=0$. So take $k\ge 1$ and assume that it holds for $k-1$. We need to prove that $\CC^{(k)}\subseteq(k+1)\CC$.
Take $\M(\be)\in\CC^{(k)}$. Then $\M(\be)$ is a genuine limit of a convergent sequence $\{\M(\be_n)\}_{n\in\N}$ say, in $\CC^{(k-1)}$. By the induction hypothesis, $\CC^{(k-1)}\subseteq\CC_{(k)}$, so that for each $\be_n$ there are cyclotomic integers $\ga_{in}\,(i=1\cc k)$ such that
\beq\label{E-betagamma}
\M(\be_n)=\M(\ga_{1n})+\M(\ga_{2n})+\cdots+\M(\ga_{kn}).
\enq
Now the sequence $\{\M(\be_n)\}$ is bounded, so the sequences $\{\M(\ga_{in})\}\,(i=1\cc k)$ are also bounded, with the same bound, $B$ say.
Thus by replacing $\{\M(\ga_{in})\}$ by an appropriate subsequence  we can assume that each $i=1\cc k$ the sequence $\{\M(\ga_{in})\}$ converges. Because the set $\CC$ is closed, the limit will be $\M(\ga_{i\infty})$, say, for some cyclotomic integer $\M(\ga_{i\infty})$. Note too that $\M(\ga_{i\infty})$ must be a genuine limit point of $\{\M(\ga_{in})\}$ for at least one value of $i$.

Further, by Loxton's  Theorem \ref{T-Loxg}, there is an integer $N'$ such that all $\ga_{in}$ can be expressed as the sum of at most $N'$ roots of unity. Hence by replacing $\{\M(\be_n)\}$ by a suitable subsequence we may assume that for each $i$ the numbers $\ga_{in}$ can be expressed as the sum of the same number, $N_i$ say, of roots of unity. By writing each $\ga_{in}$ as a sum of a minimal number $\NC(\ga_{in})$ of roots of unity, we will have $\NC(\ga_{in})=N_i$ for each $n$.

We now study one of these sequences $\{\M(\ga_{in})\}$. For this purpose we temporarily drop the `$i$' subscript, and study the convergent sequence $\{\M(\ga_{n})\}$, where each $\ga_n$ is the sum of the same number, $N$ say, of roots of unity. By replacing $\ga_n$ by an equivalent cyclotomic integer we can assume that
\beq\label{E-sumrho}
\ga_n=1+\sum_{j=2}^N\rho_{jn},
\enq
say. By re-ordering the roots of unity, if necessary, we can also assume that the orders of these roots of unity increase nonstrictly monotonically with $j$. Consider the sequence $\{\rho_{2n}\}_{n\in\N}$. If infinitely many of these roots of unity are equal, then we can replace $\{\M(\be_n)\}$ by an infinite subsequence so that all the $\rho_{2n}$'s are equal. We do the same for $\{\rho_{3n}\}$, $\{\rho_{4n}\}\cc$ until we find a $j_1$ for which $\{\rho_{j_1n}\}$
contains only finitely many copies of every root of unity. In this situation the order of $\rho_{j_1n}$ tends to infinity with $n$. We can then rewrite \eqref{E-sumrho} as
\beq\label{E-sumrho2}
\ga_n=s_0+\rho_{j_1n}+\rho_{j_1+1,n}+\cdots
\enq
where $s_0$ is a sum of roots of unity, all independent of $n$. Note that such a term $\rho_{j_1,n}$ must exist for all $i$ where $\M(\ga_{i\infty})$ is a genuine limit point of $\{\M(\ga_{in})\}$.

We now temporarily modify \eqref{E-sumrho2} to
\beq\label{E-sumrho3}
\ga_n=s_0+\rho_{j_1n}(1+\rho'_{j_1+1,n}+\rho'_{j_1+2,n}+\cdots).
\enq
say.  We then reorder the sequence $\rho'_{j_1+1,n},\rho'_{j_1+2,n},\dots,\rho'_{N,n}$ so that their orders as roots of unity are (nonstrictly) monotonically increasing.  If the sequence $\{\rho'_{j_1+1,n}\}_{n\in\N}$ has infinitely many equal terms, then we can take an infinite subsequence of $\{\M(\be_n)\}$ where $\{\rho'_{j_1+2,n}\}$ is constant. We do the same for $\{\rho'_{j_1+2,n}\}$, if possible. We continue in this way until we encounter a sequence, $\{\rho_{j_2n}\}$ say, that contains only finitely many copies of each root of unity; we then define 
\[
s_1:=1+\sum_{j=j_1+1}^{j_2-1} \rho_{jn}'
\] so that we can rewrite \eqref{E-sumrho3} as
\beq\label{E-sumrho4}
\ga_n=s_0+\rho_{j_1n}s_1+\rho_{j_2n}(1+\rho'_{j_2+1,n}+\rho'_{j_2+2,n}+\cdots).
\enq
Note that the order of $\rho_{j_2n}/\rho_{j_1n}$ tends to infinity with $n$.

Continuing in this way, we can finally write $\ga_n$ as
\beq\label{E-sumrho5}
\ga_n=s_0+\rho_{j_1n}s_1+\rho_{j_2n}s_2+\cdots+\rho_{j_rn}s_r,
\enq
where $r\ge 1$ for at least one value of $i$, and $s_0,s_0\cc s_r$ are sums of roots of unity, all independent of $n$. In general they will, of course, depend on the (dropped) subscript $i$. Also, all of the $s_k$'s must be nonzero, as $\ga_n$ has been written as the sum of a minimal number of roots of unity. Furthermore,  for $k=1\cc r$ and $\l=k+1\cc r$ the order of $\rho_{j_{\l}n}/\rho_{j_kn}$ tends to infinity with $n$. For if the sequence $\{\text{order of }\rho_{j_{\l}n}/\rho_{j_kn}\}$ were bounded, then we could assume, by the above subsequence argument, that it would be constant. Then the term $\rho_{j_{\l}n}=\rho_{j_kn}(\rho_{j_{\l}n}/\rho_{j_kn})$ would already have contributed a root of unity to $s_k$.

From \eqref{E-sumrho5} and Proposition \ref{P-bygum} we see that
$\M(\ga_n)\to \M(s_0)+\M(s_1)+\cdots+\M(s_r)$ as $n\to\infty$.  On reinstating the dropped subscript $i$, and applying this result to each sequence $\{\M(\ga_{in})\}$,
 we see that for each $i$ the limit of this sequence is a sum of $r_i:=1+r$ elements of $\CC$. We have seen above that  $r_i\ge 2$ for at least one value of $i$, so that from \eqref{E-betagamma} that $\M(\be)=\lim_{n\to\infty}\M(\be_n)$ is a sum of  $k+t$ elements of $\CC$, where $t\ge 1$. Since by additivity
  we can express a sum of $t$ elements of $\CC$ as a single element of $\CC$, we have $\M(\be)\in\CC_{k+1}$, as required.
\end{proof}

This completes the proof of Theorem \ref{T-1}. We now know that $\CC$ is a countable closed set, having nonempty derived sets of all orders $k$, with every element of $\CC^{(k)}$ being a two-sided limit of elements of $\CC^{(k-1)}$,  and with the smallest element of the $k$th derived set tending to infinity as $k$ goes to infinity. Thus $\CC$ is a Thue set, proving Corollary \ref{C-2}.

\subsection{Structure and labelling of Thue sets.}

Two totally ordered sets are said to have the {\it same order type} if there is an order-preserving bijection between them. The {\it order type} of a set is then the ordinal having the same order type as the set.  For the ordinal $\om$, put $a_1=\om+1+\om^*$, and $a_{n+1}=a_n\om+1+(a_n\om)^*$ for $n\ge 1$. Here $()^*$ denotes the reverse order. Boyd and Mauldin \cite{BoydMauldin1996} showed that the order type of the set of PV numbers is $\sum_{n=1}^\infty a_n$. 

 Let $T$ be any Thue set. We will now build a finite string of integers to label a given element $t$ of $T$. We proceed as follows. If $t<t_1$ then $t$ is an element of the increasing sequence of all members of $T$ that are less than $t_1$, which we label $\l_{00},\l_{01},\l_{02},\dots$. For $t\ge t_1$ choose the largest $k$ such that $t \ge t_k$. Take $k$ as the first element of our string. Then there are no limit points of $T^{(k)}$ (i.e., elements of $T^{(k+1)}$) that are less than $t$, so that  $T^{(k)}$ is discrete in the interval $[t_k,t_{k+1})$, which must contain $t$. We label the elements of $[t_k,t_{k+1})\cap T^{(k)}$ in ascending order by $\l_{k0},\l_{k1},\l_{k2},\dots$. Then $t$ is in one of the half-open intervals $[\l_{kr},\l_{k,r+1})$ say; we take $r$ to be the second element of our string.
 If $t=\l_{kr}$, end the string. Otherwise, we note that 
  the elements of $T^{(k-1)}$ in the interval $(\l_{kr},\l_{k,r+1})$ form a countable set with limit points precisely at both endpoints of the interval. For definiteness we label those in $[\tfrac12(\l_{kr}+\l_{k,r+1}),\l_{k,r+1})$ by $\l_{kr0},\l_{kr1},\l_{kr2},\dots$ in ascending order, and those in $[\l_{kr},\tfrac12(\l_{kr}+\l_{k,r+1}))$ by $\l_{kr,-1},\l_{kr,-2},\l_{kr,-3},\dots,$ in descending order.
 Again, $t$ is in one of the half-open intervals defined by these points,
 so we label it by the left endpoint. Again, if $t$ is equal to this endpoint, the label ends. Otherwise, we note that in the open interval there is a countable ascending string of elements of
 $T^{(k-2)}$ with limit points precisely at both endpoints of the interval. So we can proceed as before. Continuing in this way,
  the string ends by $t$ being a left endpoint of an interval (the elements with the longest strings will be those $t$ in an interval 
  whose endpoints are in $T\setminus T^{(1)}$. Then $t$ must equal the left endpoint of such an interval. Thus, in the end, every element of $T$ is of the form $\l_s$, where $s$ is a string of integers, which we call the {\it label} of $\l_s$; we have seen that $s$ is of the form  $s=k r_1\cdots r_j$,  where $k\ge 0$ and  $1\le j\le k+1$. This tells us that $t_k\le t<t_{k+1}$ and that
 $ t\in T^{(k-j+1)}\setminus T^{(k-j+2)}$.
  
  The labelling described is ordered by the most significant digits, with the added rule that if two strings are of different lengths, but agree for the whole length of the shorter one, then this shorter one comes first in the ordering. Then this ordering coincides with the ordering on the real line.
  
  Note that the allowable integer string labels are subject to the following constraints:
  \begin{itemize}
  \item The first term, $k$, is non-negative;
  \item If $k=0$ then the second term is non-negative;
  \item The string must contain between $2$ and  $k+2$ terms.
  \eni

  \begin{prop} \label{P-Thue} Any two Thue sets have the same order type.
  \end{prop}

\section{Proof of Theorem \ref{T-quads}.}

\begin{proof}
First, note that for any integer $m\ge 0$
$$
a^2+ab+b^2+c^2+a+b+c=m
$$
has an integer solution if and only if
$$
3(2a+b+1)^2+(3b+1)^2+3(2c+1)^2=12m+7
$$
has an integer solution.  Note that the class number of $x^2+3y^2+3z^2$ is one by \cite{Jones1935}, and by using  \cite[\S 102.5]{O'Meara1963} and \cite{O'Meara1958},  one may easily check that there are integers $x,y$, and $z$ such that 
$$
x^2+3y^2+3z^2=12m+7.
$$  
Since $x$ is not divisible by $3$, by changing,  if necessary, the sign of $x$, there is an integer $b$ such that $3b+1=x$.  Assume that $x$ is even. Then $b$ is odd.
 In this case, since $y-z$ is odd, without loss of generality, we may assume that $y$ is odd. Therefore there are integers $a$ and $c$ such that 
$$
2a+b+1=z, \quad 2c+1=y.
$$
Now, assume that $x$ is odd.  Then $b$ is even and both $y$ and $z$ are odd. Therefore there are integers $a$ and $c$ satisfying the above.  
Thus \eqref{E-qp1} is universal, as claimed.

Next, note that any integer $m\ge0$, the equation 
$$
a^2+b^2+c^2+ab+bc+ca+a+b+c=m
$$
has an integer solution if and only if
$$
6(2a+b+c+1)^2+2(3b+c+1)^2+(4c+1)^2=24m+9
$$
has an integer solution.  Note that the class number of $x^2+2y^2+6z^2$ is one by \cite{Jones1935}, and by again using   \cite[\S 102.5]{O'Meara1963},  one may easily check that there are integers $x,y$, and $z$ such that 
$$
x^2+2y^2+6z^2=24m+9.
$$  
 Since $x$ is  odd, by changing the sign of $x$, if necessary, there is an integer $c$ such that $4c+1=x$.  Note that $x$ is divisible by $3$ if and only if $y$ is divisible by $3$. Hence there is an integer $b$ such that $3b+c+1=y$ by changing, if necessary, the sign of $y$. Finally, since $y\equiv z\pmod{2}$, there is an integer $a$ such that $2a+b+c+1=z$.  Thus \eqref{E-qp2} is universal.
 
\end{proof}

\section{Proof of Theorem \ref{T-2}.}

\begin{proof}
For $p$ an odd prime, let $\be = \sum_{i=0}^{p-1} a_i \om_p^i\in\Z[\om_p]$.
 These coefficients $a_i$ are not uniquely determined by $\beta$: we can replace each $a_i$ by $a_i+t$ for any $t\in \Z$. 
 Thus we can assume that $s:=\sum_{j=0}^{p-1}a_j\in[-p',p']$, where $p':=(p-1)/2$. In fact, since $\M(-\be)=\M(\be)$, we can assume for the study of $\CC_p$ that 
$s$ is an integer in $[0,p']$. Also write $\mathrm{var}(a_0\cc a_{p-1})$ for the variance of $a_0\cc a_{p-1}$.  We need the following.
\begin{lemma}\label{L:variance}
We have
\begin{align}
    p'\M(\be) &= \frac12\left(p\sum_{j=0}^{p-1}a_j^2-s^2\right)   \label{E-p'M} \\
               &=\frac{p^2}{2}\mathrm{var}(a_0,\dots,a_{p-1})\,.\label{E-var}
\end{align}
\end{lemma}
\begin{proof}[Proof of Lemma \ref{L:variance}]
We have
\begin{align*}
p'\M(\be) &=\tfrac12\sum_{i=1}^{p-1}\left(\sum_{j=0}^{p-1}a_j\om_p^{ij}\sum_{k=0}^{p-1}a_k\om_p^{-ik}\right) \\
&= \tfrac12\sum_{j=0}^{p-1}\sum_{k=0}^{p-1}a_ja_k\left(\sum_{i=1}^{p-1}\om_p^{i(j-k)}\right) \\
&= \tfrac12\sum_{j=0}^{p-1}\sum_{k=0}^{p-1}a_ja_k\left(\sum_{i=0}^{p-1}\om_p^{i(j-k)}-1\right) \\
&=\tfrac12\left(p\sum_{j=0}^{p-1}a_j^2-s^2\right),
\end{align*}
giving \eqref{E-p'M}. This also equals
\[
\tfrac{p^2}{2}\left(\tfrac1{p}\sum_{j=0}^{p-1}a_j^2-\left(\tfrac{s}{p}\right)^2\right)=\tfrac{p^2}{2}\mathrm{var}(a_0,\dots,a_{p-1}).
\]
\end{proof}
Thus we can interpret the Cassels height $\M(\be)$ as a fixed multiple (depending on $p$) of the variance of the sequence of coefficients $a_i$ of $\be$. Thus, for $p$ and $s$ given, $p'\M(\be)$ is minimised when the $a_i$ are as close as possible to each other (and to their mean, which lies in $[0,1/2)$), and the minimum occurs precisely when $s$ of the $a_i$ equal $1$, while the remaining $p-s$ are $0$. From the formula \eqref{E-p'M}, 
we see that this minimum of $p'\M(\be)$ is $\frac{s(p-s)}{2}$.
 Furthermore, up to permutation of the $a_k$'s, 
this is the only sequence for which the minimum occurs.

We must now show that for  $p'\M(\be)$ can take all integer values  $\frac{s(p-s)}{2}+rp$, for all integers $r\ge 0$. We separate three cases.

\begin{itemize}[leftmargin=*]
\item $p\ge 11$.
From the $s$ ones and $p-s$ zeroes in $a_0,a_1\cc a_{p-1}$ we can choose $\lfloor\tfrac{s}{2}\rfloor+\lfloor\tfrac{p-2}{2}\rfloor> 4$ pairs of equal values (both $1$ or both $0$). Taking four of these pairs $(a,a)$ and replacing each by $(a+n,a-n)$ for some integer $n$, we see from \eqref{E-p'M} that $p'\M(\be)$ is increased by $p$ times the sum of four squares of integers. Since every nonnegative integer $r$ is the sum of four squares \cite{Lagrange1770}, we indeed have that $p'\M(\be)$ can take every value $\frac{s(p-s)}{2}+rp$.

\item $p=7$. We have $s=0,1,2$ and $3$. Let $\ua=(0,0,0,0,0,a_5,a_6)$, so that $s=a_5+a_6$. Now change $\ua$ to  $\ua=(-a,-b,-c,-d,a+b+c+d,a_5,a_6)$. Then  $s$ remains unchanged, while $\sum_{j=0}^{6}a_j^2$ increases by $7/2$ times
\beq\label{E-quad5}
a^2+b^2+c^2+d^2+(a+b+c+d)^2=2(a^2  + b^2 + c^2 + d^2+ a(b+c+d)+b(c+d)+cd).
\enq
  This quadratic form, with root lattice $A_4$, has class number $1$ (see Nipp  \cite{Nipp1991}), and locally represents all even integers. Hence by [2, \S102.5], it represents all even positive integers.   
 By choosing $a_5,a_6=0$ or $1$, and so $s=0,1$ or $2$ we see from \eqref{E-p'M} that $p'\M(\be)$ can take all values $\tfrac12 s(7-s)+7r$ for every integer $r\ge 0$, for these values of $s$. For $s=3$  and $\ua=(0,0,0,0,1,1,1)$ we change $\ua$ to $\ua=(a,b,c,-(a+b+c),d+1,-d+1,1)$.
Here still $s=3$, while $\sum_{j=0}^{p-1}a_j^2$ increases by $7/2$ times
\[
a^2+b^2+c^2+(a+b+c)^2+(d+1)^2+(-d+1)^2=2(a^2 + (b + c)a + b^2 + bc + c^2 + d^2).
\]
 This quadratic form, with root lattice $A_3\perp A_1$, has class number $1$ (see \cite{Nipp1991}), and locally represents all even integers. Hence by [2, \S102.5], it represents all even positive integers.   
Thus $p'\M(\be)$ can take all values $\tfrac12 s(7-s)+7r$ for every integer $r\ge 0$ for $s=3$ also.

\item $p=5$. We have $s=0,1$ and $2$. The case $s=0$ is essentially the same as for $p=7$: take $\ua=(-a,-b,-c,-d,a+b+c+d)$, with again $\sum_{j=0}^{4}a_j^2$ given by \eqref{E-quad5}. For $s=1$, start with $\ua=(0,0,0,0,1)$ and change it to $(0,-a,-b,-c,1+a+b+c)$. Then $\sum_{j=0}^{4}a_j^2$ increases by $5/2$ times
\[
a^2+b^2+c^2+(1+a+b+c)^2-1=2(a^2+b^2+c^2+ab+bc+ca+a+b+c).
\]
Hence, by Theorem \ref{T-quads}, $p'\M(\be)$ can take all values $2+5r$ for every integer $r\ge 0$. For $s=2$, start with $\ua=(0,0,0,1,1)$ and change it to $(-a,-b,-c,1+a+b,1+c)$. Then $\sum_{j=0}^{4}a_j^2$ increases by $5/2$ times
\[
a^2+b^2+c^2+(1+a+b)^2+(1+c)^2-2=2(a^2+ab+b^2+c^2+a+b+c).
\]
Hence, again by Theorem \ref{T-quads}, $p'\M(\be)$ can take all values $3+5r$ for every integer $r\ge 0$.

\eni
\end{proof}

Note that it follows that, for $\be\in\Z[\om_p]$, $\M(\be)$ depends only on the set $\{a_k\}$ of coefficients of $\be$, and not on their order. (In fact it depends only on $\sum_k a_k$ and $\sum_k a_k^2$.) Thus in general there are many inequivalent $\be\in\Q(\om_p)$ with the same value of $\M(\be)$. Note too that, having established Theorem \ref{T-2}, Lemma \ref{L:variance} provides a description of the possible values of the variance of a sequence of $p$ integers.

\bibliography{MOSmain}

\end{document}